\documentclass[reqno,12pt]{amsart}
\usepackage{amsmath,amsthm,amssymb,amsfonts,amscd}
\usepackage{epsfig}
\usepackage{color}
\usepackage{xypic}
\numberwithin{equation}{section}

\theoremstyle{plain}
\newtheorem{theorem}{Theorem}
\newtheorem{conjecture}[theorem]{Conjecture}

\newtheorem{corollary}[theorem]{Corollary}
\newtheorem{proposition}[theorem]{Proposition}
\newtheorem*{theorem*}{Theorem}
\newtheorem*{conjecture*}{Conjecture}

\theoremstyle{definition}
\newtheorem{remark}[theorem]{Remark}

\newtheorem*{definition}{Definition}

\newcommand{\CC}{{\mathbb{C}}}

\newcommand{\QQ}{{\mathbb{Q}}}

\newcommand{\RR}{{\mathbb{R}}}
\newcommand{\ZZ}{{\mathbb{Z}}}
\newcommand{\LL}{{\mathbb{L}}}

\newcommand{\NN}{{\mathbb{N}}}

\newcommand{\eps}{\varepsilon}

\def\H{{\mathcal H}}

\begin{document}
\title[Lattices for Landau--Ginzburg orbifolds]{Lattices for Landau--Ginzburg orbifolds}
\author{Wolfgang Ebeling and Atsushi Takahashi}
\address{Institut f\"ur Algebraische Geometrie, Leibniz Universit\"at Hannover, Postfach 6009, D-30060 Hannover, Germany}
\email{ebeling@math.uni-hannover.de}
\address{
Department of Mathematics, Graduate School of Science, Osaka University, 
Toyonaka Osaka, 560-0043, Japan}
\email{takahashi@math.sci.osaka-u.ac.jp}
\date{}

\begin{abstract} 
We consider a pair consisting of an invertible polynomial and a finite abelian group of its symmetries. Berglund, H\"ubsch, and Henningson proposed a duality between such pairs giving rise to mirror symmetry. We define an orbifoldized signature for such a pair using the orbifoldized elliptic genus. In the case of three variables and based on the homological mirror symmetry picture, we introduce two integral lattices, a transcendental and an algebraic one. We show that these lattices have the same rank and that the signature of the transcendental one is the orbifoldized signature. Finally, we give some evidence that these lattices are interchanged under the duality of pairs.
\end{abstract}
\maketitle
\section*{Introduction}
Let $X$ be a smooth compact K\"ahler manifold of dimension $n$. 
The Hodge numbers $h^{p,q}(X):=\dim_\CC H^q(X,\Omega_X^p)$, $p,q \in \ZZ$, are some  of the most important numerical invariants 
of $X$. 
They satisfy
\[
h^{p,q}(X)=h^{q,p}(X),\quad p,q\in\ZZ,
\]
and the Serre duality
\[
h^{p,q}(X)=h^{n-p,n-q}(X),\quad p,q\in\ZZ.
\]
Let $n$ be even. Then the intersection form on the free part of the middle homology group  $H^n(X, \ZZ)/{\rm Tors}(H^n(X, \ZZ))$ is a non-degenerate symmetric bilinear form. Let $\mu_+$ and $\mu_-$ be the number of positive and negative eigenvalues of this form respectively. Then one has the following well-known formulas for these numbers:
\[
\mu_+ = \sum_{q \, {\rm even} \atop p+q \equiv 0\, (2)} h^{p,q}(X), \quad \mu_+ = \sum_{q \,{\rm odd} \atop p+q \equiv 0\, (2)} h^{p,q}(X).
\]

There are analogues of these formulas for the Milnor fibre of a weighted homogeneous polynomial $f(x_1,\dots, x_n)$, $n$ odd, with an isolated singularity at the origin.
There, the analogue of the multi-set $\{q*h^{p,q}(X)~|~p,q\in \ZZ,\ h^{p,q}(X)\ne 0 \}$, where by $u*v$ we denote $v$ copies of the integer $u$, above will be the  {\em set of exponents} of $f(x_1,\dots, x_n)$,  
which is a set of rational numbers and is also one of the most important numerical invariants defined by the mixed Hodge structure 
associated to $f(x_1,\dots, x_n)$. 
Let $f(x_1,\dots, x_n)$ be a non-degenerate weighted homogeneous polynomial, namely, a polynomial with an isolated singularity at the origin with the property that 
there are positive rational numbers $w_i$, $i=1,\dots ,n$, such that
 \[ f(\lambda^{w_1}x_1,\dots ,\lambda^{w_n}x_n)=\lambda f(x_1,\dots,x_n), \quad \lambda\in\CC\backslash\{0\}. \]
Let $\{ z^\beta \, | \, \beta \in I \subset \NN^n \}$ be a set of monomials in $\CC[z_1, \ldots, z_n]$ which represent a basis of the finite dimensional vector space $\CC[z_1, \ldots, z_n]/(\frac{\partial f}{\partial z_1}, \ldots , \frac{\partial f}{\partial z_n})$. For $\beta \in I$, let $l(\beta)=\sum_{i=1}^n (\beta_i +1) w_i$.
Let $V_f=f^{-1}(1)$ be the Milnor fibre of $f$.

One has the following theorem due to J.~H.~M.~Steenbrink.

\begin{theorem}[cf.\ \cite{st:2}] \label{thm:st}
Let $\mu_+$, $\mu_-$, and $\mu_0$ be the number of positive, negative and zero eigenvalues of the intersection form on $H_{n-1}(V_f, \RR)$ respectively. Then
\begin{eqnarray*}
\mu_+ & = & \sharp \{ \beta \in I \, | \, l(\beta) \not\in \ZZ \mbox{ and } [l(\beta)] \mbox{ even} \}, \\
\mu_-  &  = & \sharp \{ \beta \in I \, | \, l(\beta) \not\in \ZZ \mbox{ and } [l(\beta)] \mbox{ odd} \}, \\
\mu_0 & = & \sharp \{ \beta \in I \, | \, l(\beta) \in \ZZ \}.
\end{eqnarray*}
\end{theorem}

A special case of a non-degenerate weighted homogeneous polynomial is a so-called {\em invertible polynomial}. This is a polynomial which has as many terms as there are variables. P.~Berglund and T.~H\"ubsch \cite{BHu} considered a duality between such polynomials which is related to mirror symmetry: to an invertible polynomial $f$ is associated its transpose $\widetilde{f}$.  Berglund and M.~Henningson \cite{BHe} generalized this to an orbifold setting. A {\em Landau--Ginzburg orbifold} is a pair $(f,G)$ where $f$ is an invertible polynomial and $G$ a (finite) abelian group of symmetries of $f$. To such a pair they associate a dual pair $(\widetilde{f}, \widetilde{G})$. One can assign a set of exponents to such a pair, see \cite{ET1, ET2, EGT}. The generating function of this set of exponents is the orbifold E-function. In joint work with S.~M.~Gusein-Zade \cite{EGT}, it was shown that there is a symmetry property of the orbifold E-functions of Berglund--H\"ubsch--Henningson dual pairs. 

Here we derive certain properties from this symmetry property. Namely, we first show that this yields the symmetry of the orbifoldized elliptic genera as predicted by Berglund and Henningson \cite{BHe} and T.~Kawai and S.-K.~Yang \cite{KY}. We show how to obtain the signature of the Milnor lattice from the unorbifoldized elliptic genus. This gives a definition of the signature for the orbifold case. This leads to the question of the realization of a lattice with this signature in the orbifold case. In the case $n=3$, we offer two lattices, one for the A-model in the case $G \subset {\rm SL}(3;\CC) \cap G_f$, the other one for the dual B-model in the case $G_0 \subset G$, where $G_0$ is the subgroup of the group $G_f$ of all diagonal symmetries generated by the exponential grading operator. In both cases, we show that these lattices have the correct rank. Moreover, in the case $G \subset {\rm SL}(3;\CC)$, we show that it has the correct signature. 
We give some evidence that these lattices are the correct generalizations of the Milnor lattice.

\begin{sloppypar}

{\bf Acknowledgements}.\  
This work has been partially supported by DFG.
The second named author is also supported 
by JSPS KAKENHI Grant Number 16H06337. 
\end{sloppypar}

\section{Landau--Ginzburg orbifolds}
In this section we collect some definitions and facts about Landau---Ginzburg orbifolds.

A non-degenerate weighted homogeneous polynomial $f(x_1,\ldots ,x_n)$ is called {\em invertible} if 
the following conditions are satisfied:
\begin{enumerate}
\item the number of variables ($=n$) coincides with the number of monomials 
in the polynomial $f(x_1,\ldots x_n)$, 
namely, 
\[
f(x_1,\ldots ,x_n)=\sum_{i=1}^n c_i\prod_{j=1}^nx_j^{E_{ij}}
\]
for some coefficients $c_i\in\CC^\ast$ and non-negative integers 
$E_{ij}$ for $i,j=1,\ldots, n$,
\item a system of weights $(w_1,\ldots ,w_n;1)$ can be uniquely determined by 
the polynomial $f(x_1,\ldots ,x_n)$, 
namely, the matrix $E:=(E_{ij})$ is invertible over $\QQ$.
\end{enumerate}

Let $f(x_1,\ldots ,x_n)=\sum_{i=1}^n c_i\prod_{j=1}^nx_j^{E_{ij}}$ be an invertible polynomial.
The {\em group of diagonal symmetries} $G_f$ of $f$ is
the finite abelian group defined by 
\[
G_f=\left\{(\lambda_1,\ldots ,\lambda_n)\in (\CC^\ast)^n\, \left| \,
f(\lambda_1 x_1, \ldots, \lambda_n x_n) = f(x_1,\ldots ,x_n) \right\} \right..
\]
In other words
\[
G_f=\left\{(\lambda_1,\ldots ,\lambda_n)\in (\CC^\ast)^n\, \left| \,
\prod_{j=1}^n \lambda_j ^{E_{1j}}=\ldots =\prod_{j=1}^n\lambda_j^{E_{nj}}=1 \right\} \right..
\]

Note that $G_f$ always contains the exponential grading operator 
\[
g_0:=({\bf e}[w_1], \ldots , {\bf e}[w_n]),
\]
where ${\bf e}[ - ] := e^{2 \pi \sqrt{-1} \cdot  -}$. Denote by $G_0$ the subgroup of $G_f$ generated by $g_0$.

Let $f(x_1, \ldots , x_n)$ be an invertible polynomial and $G$ a subgroup of $G_f$. The pair $(f,G)$ is called a {\em Landau--Ginzburg orbifold}.
For $g \in G$,
we denote by ${\rm Fix}\, g :=\{x\in\CC^n~|~g\cdot x=x \}$ the fixed locus of $g$, by
$n_g: = \dim {\rm Fix}\, g$ its dimension and by $f^g:=f|_{{\rm Fix}\, g}$ the restriction of $f$
to the fixed locus of $g$. Note that the function $f^g$ has an isolated singularity at the origin
\cite[Proposition~5]{ET2}.

Y.~Ito and M.~Reid \cite{IR} introduced the notion of the age of an element of a finite group as follows:
Let $g \in G$ be an element and $r$ be the order of $g$. Then $g$ has a unique expression of the following form  
\[
g=\left({\bf e}\left[\frac{p_1}{r}\right], \ldots, {\bf e}\left[\frac{p_n}{r}\right]\right)
\quad \mbox{with } 0 \leq p_i < r.
\]
(Such an element $g$ is often simply denoted by $g=\frac{1}{r}(p_1, \ldots, p_n)$.)
The {\em age} of $g$ is defined as \[ {\rm age}(g) := \frac{1}{r}\sum_{i=1}^n p_i. \] 

Berglund and H\"{u}bsch \cite{BHu} defined the dual polynomial $\widetilde{f}$ of $f$ as follows.
Let $f(x_1,\ldots ,x_n)$ be an invertible polynomial. 
The {\em Berglund--H\"{u}bsch transpose} $\widetilde{f}(x_1,\ldots ,x_n)$ 
of $f(x_1,\ldots ,x_n)$ is defined by
\begin{equation}
\widetilde{f}(x_1,\ldots ,x_n):=\sum_{i=1}^n c_i\prod_{j=1}^nx_j^{E_{ji}}.
\end{equation}

Berglund and Henningson \cite{BHe} defined the {\em dual group} $\widetilde{G}$ of $G$ as follows:
\begin{equation}
\widetilde{G} := {\rm Hom}(G_f/G, \CC^\ast).
\end{equation}
One  can see that $G_0\subset G$ if and only if
$\widetilde{G}\subset {\rm SL}(n;\CC)\cap G_{\widetilde{f}}$.

\section{Orbifoldized elliptic genus}
Let $f(x_1,\dots, x_n)$ be a non-degenerate weighted homogeneous polynomial with weights $w_1, \ldots , w_n$ and degree $d$.
Let $\tau$ be an element of the upper half plane, $z \in \CC$, $q=e^{2 \pi \sqrt{-1} \tau}$, and $y=e^{2 \pi \sqrt{-1} z}$. The {\em elliptic genus} $Z(\tau, z)$ is defined as follows (cf.\ the papers of Kawai, Y.~Yamada and Yang \cite{KYY} and Kawai and Yang \cite{KY}):
\[ 
Z(\tau, z) := \prod_{i=1}^n Z_{w_i}(\tau, z),
\]
where
\begin{eqnarray*}
Z_{w_i}(\tau, z) & = & y^{-(1-2w_i)/2} \frac{(1-y^{1-w_i})}{(1-y^{w_i})} \prod_{l=1}^\infty \frac{(1-q^ly^{1-w_i})}{(1-q^ly^{w-i})} \frac{(1-q^ly^{-(1-w_i)})}{(1-q^ly^{-w_i})} \\
 & = & \frac{\vartheta_1(\tau, (1-w_i)z)}{\vartheta_1(\tau, w_i z)}
\end{eqnarray*}
and $\vartheta_1(\tau, z)$  is one of the Jacobi theta functions, namely
\[
\vartheta_1(\tau,z) = \sqrt{-1} q^{1/8}y^{-1/2} \prod_{l=1}^\infty (1-q^l)(1-q^{l-1}y)(1-q^ly^{-1}).
\]

Let $G \subset G_f$. Then the {\em orbifoldized elliptic genus} is defined as follows (cf.\ \cite{KYY, KY}). Let $a=(a_1, \ldots , a_n)$ stand for the element 
$( {\bf e}[a_1 w_1], \ldots , {\bf e}[a_n w_n])$ of $G$. Then
\[
Z(f,G)(\tau,z) := \frac{1}{|G|} \sum_{a, b \in G} \prod_{i=1}^n (-1)^{a_i + b_i + a_i b_i} Z_{a b}(f,G)(\tau, z),
\]
where
\begin{eqnarray*}
\lefteqn{Z_{a b}(f,G)(\tau, z)} \\
 & = & \prod_{i=1}^n {\bf e}\left[ \frac{1-2w_i}{2} a_i b_i \right] {\bf e} \left[ \frac{1-2w_i}{2}(a_i^2\tau+2a_iz) \right] 
\frac{\vartheta_1(\tau, (1-w_i)(z+a_i\tau+b_i))}{\vartheta_1(\tau, w_i(z+a_i\tau +b_i))}.
\end{eqnarray*}

The elliptic genus specializes to the {\em $\chi_y$-genus}. Let 
\[
\hat{c} := \sum_{i=1}^n (1-2w_i).
\]
Then 
\[
\chi(f,G)(y) = y^{\hat{c}/2} Z(f,G)(\sqrt{-1} \infty,z),
\]
where $Z(f,G)(\sqrt{-1} \infty,z) = \lim_{\tau \to \sqrt{-1}\infty} Z(f,G)(\tau,z)$.
One obtains the following formula (cf.\ \cite{ET2}):
\[ \chi(f,G)(y) = (-1)^n \sum_{a \in G}\left( y^{{\rm age}(a) - \frac{n-n_a}{2}}
\cdot \frac{1}{|G|} \sum_{b \in G}\prod_{{\bf e}[a_iw_i] = 1}
\frac{y^{\frac{1}{2}}-{\bf e}[b_iw_i]y^{w_i-\frac{1}{2}}}
{1-{\bf e}[b_iw_i]y^{w_i}} \right) .
\]

Now let $f$ be an invertible polynomial and let $G$ be a subgroup of $G_f$ with $G \subset {\rm SL}(3;\CC) \cap G_f$ or $G_0 \subset G$. We consider the Landau--Ginzburg orbifold $(f,G)$ and the dual pair $(\widetilde{f},\widetilde{G})$. Let $E(f,G)(t, \bar{t})$ be the orbifold E-function defined in \cite{ET1, ET2, EGT}. We recall the definition of this function.
Steenbrink \cite{st:1} constructed a canonical mixed Hodge structure on the vanishing cohomology $H^{n-1}(V_f,\CC)$ of $f$
with an automorphism $c$ given by the Milnor monodromy. 
In our situation, the automorphism $c$ coincides with the one induced by the action of $g_0\in G_f$ on $\CC^n$, 
which is semi-simple.
For $\lambda\in \CC$, let  
\begin{equation*}
H^{n-1}(V_f;\CC)_\lambda:={\rm Ker}(c-\lambda\cdot {\rm id}:H^{n-1}(V_f;\CC)\longrightarrow 
H^{n-1}(V_f;\CC)).
\end{equation*}
Denote by $F^\bullet$ the Hodge filtration of the mixed Hodge structure.

We can naturally associate a $\QQ\times \QQ$-graded vector space to the mixed Hodge structure 
with an automorphism. 
Define the $\QQ\times \QQ$-graded vector space $\H_f:=\displaystyle\bigoplus_{p,q\in\QQ}\H^{p,q}_f$ as
\begin{enumerate}
\item If $p+q\ne n$, then  $\H^{p,q}_f:=0$.
\item If $p+q=n$ and $p\in\ZZ$, then  
\[
\H^{p,q}_f:={\rm Gr}^{p}_{F^\bullet}H^{n-1}(V_f,\CC)_1.
\]
\item If $p+q=n$ and $p\notin\ZZ$, then  
\[
\H^{p,q}_f:={\rm Gr}^{[p]}_{F^\bullet}H^{n-1}(V_f,\CC)_{e^{2\pi\sqrt{-1} p}},
\]
where $[p]$ is the largest integer less than $p$.
\end{enumerate}
Note that $\H_f \cong H^{n-1}(V_f,\QQ) \otimes_{\QQ} \CC$. This a $\frac{1}{d} \ZZ$-indexed {\em pure} Hodge structure of weight $n$ in the sense of T.~Yasuda \cite[Sect.~3.8]{Ya}.
Define the bi-graded $\CC$-vector space $\H_{f,G}$ as 
\begin{equation*}
\H_{f,G}:=\bigoplus_{g\in G}(\H_{f^g})^G(-{\rm age}(g)).
\end{equation*}
This is a $\frac{1}{d} \ZZ$-indexed {\em mixed} Hodge structure in the sense of \cite{Ya}.

\begin{definition}
The {\em Hodge numbers} for the pair $(f,G)$ are  
\[
h^{p,q}(f,G):= \dim_\CC  \H^{p,q}_{f,G},\quad p,q\in\QQ.
\]
\end{definition}
\begin{definition}
The rational number $q$ with $\H^{p,q}_{f,G}\ne 0$ is called an {\em exponent} of the pair $(f,G)$. 
The {\em set of exponents} of the pair $(f,G)$ is the multi-set of exponents 
\[
I(f,G):=\left\{q*h^{p,q}(f,G)~|~p,q\in\QQ,\ h^{p,q}(f,G)\ne 0 \right\},
\]
where by $u*v$ we denote $v$ copies of the rational number $u$.  
\end{definition}
The orbifold E-function $E(f,G)(t, \bar{t})$ is the generating function of these exponents. For a precise definition see \cite{EGT}.
We have
\[
\chi(f,G)(y)=E(f,G)(1, y)
\]
(cf. \cite{ET2}). Therefore \cite[Theorem 9]{EGT} gives
\begin{theorem} \label{thm:chiy}
Let $f(x_1, \ldots , x_n)$ be an invertible polynomial and $G$ a subgroup of $G_f$. Then 
\begin{equation*} 
\chi(f,G)(y) = (-1)^n \chi(\widetilde{f},\widetilde{G})(y).
\end{equation*}
\end{theorem}

\begin{corollary} \label{cor:ellgen}
Let $f(x_1, \ldots , x_n)$ be an invertible polynomial and $G$ a subgroup of $G_f$. Then 
\begin{equation*} 
Z(f,G)(\tau,z) = (-1)^n Z(\widetilde{f},\widetilde{G})(\tau,z).
\end{equation*}
\end{corollary}

\begin{proof}
Let us consider the function
\[ \phi(\tau,z) := \frac{(-1)^nZ(f,G)}{Z(\widetilde{f},\widetilde{G})}(\tau,z).
\]
This is a meromorphic function in $z$.
By \cite[(4.10)]{KYY}, we have 
\begin{equation} \label{eq:per_z}
\phi(\tau, z + \lambda \tau + \mu) = \phi(\tau,z)
\end{equation}
for $\lambda , \mu \in |G_f|\ZZ=|G_{\widetilde{f}}|\ZZ$. Moreover, by \cite[(4.9)]{KYY}, we have
\begin{eqnarray}
\phi \left( \frac{-1}{\tau} , \frac{-z}{\tau} \right) & = & \phi(\tau,z), \label{eq:S} \\
\phi(\tau+1,z) & = & \phi(\tau,z). \label{eq:T}
\end{eqnarray}
Equations~(\ref{eq:per_z}), (\ref{eq:S}), and (\ref{eq:T}) imply that $\phi$ is constant.
By Theorem~\ref{thm:chiy}, we have
\[
\lim_{\tau \to \sqrt{-1} \infty} \phi(\tau,z) =1.
\]
Therefore $\phi$ is identically equal to one (cf.\ also the paper of P.~Di Francesco and S.~Yankielowicz \cite[Appendix A, Lemma]{DY}).
\end{proof}

\section{Orbifoldized signature}
Let $f(x_1, \ldots, x_n)$ be a non-degenerate weighted homogeneous polynomial with weights $w_1, \ldots, w_n$ and let $n$ be odd. We have the following reformulation of \cite[Sect.~5, Corollary]{st:2}:

\begin{theorem} \label{thm:st:ourform}
Let $\mu_+$, $\mu_-$, and $\mu_0$ be the number of positive, negative and zero eigenvalues of the intersection form on $H_{n-1}(V_f, \RR)$ respectively. Then
\begin{eqnarray*}
\mu_+ & = & \sharp \{ q \in I(f,\{{\rm id}\}) \, | \, 0< q < 1 \mbox{ mod } 2\} = \sum_{p+q=n, \atop 0 < q <1 \, {\rm  mod} \, 2} h^{p,q}(f, \{ {\rm id} \}),\\
\mu_-  &  = & \sharp \{ q \in I(f,\{{\rm id}\})  \, | \, 1< q < 2 \mbox{ mod } 2\} = \sum_{p+q=n, \atop1 < q <2 \, {\rm  mod} \, 2} h^{p,q}(f, \{ {\rm id} \}) ,\\
\mu_0 & = & \sharp \{ q \in I(f,\{{\rm id}\})  \, | \, q \in \ZZ \} = \sum_{p+q=n, \atop q \in \ZZ} h^{p,q}(f, \{ {\rm id} \}).
\end{eqnarray*}
\end{theorem}

Let  ${\rm sign}\, f = \mu_+ - \mu_-$ be the signature of the Milnor fibre $V_f$ of $f$.  Then the signature of the Milnor fibre can be computed from the elliptic genus $Z(\tau, z)$ as follows:

\begin{proposition} \label{prop:sign}
We have
\[ {\rm sign}\, f = \frac{2}{\pi} \sum_{k=- \infty \atop k\, {\rm odd}}^\infty \frac{(-1)^{\frac{kn-1}{2}}}{k} Z(\sqrt{-1} \infty, \frac{k}{2}) .
\]
\end{proposition}

\begin{proof} Let $h: \RR \to \{-1,0,+1\}$ be the function defined by
\[
h(x) = \left\{ \begin{array}{cl} (-1)^{[x]} & \mbox{for } x \not\in \ZZ, \\ 0 & \mbox{for } x \in \ZZ. \end{array} \right. 
\]
By Theorem~\ref{thm:st:ourform}, we have 
\begin{eqnarray*}
{\rm sign}\, f &=& \sharp \{ q \in I(f,\{{\rm id}\}) \, | \, 0< q < 1 \mbox{ mod } 2\} - \sharp \{ q \in I(f,\{{\rm id}\}) \, | \, 1< q < 2 \mbox{ mod } 2\}\\
&=&\sum_{p+q=n, \atop 0 < q <1 \, {\rm mod} \, 2} h^{p,q}(f, \{ {\rm id} \})-\sum_{p+q=n, \atop1 < q <2 \, {\rm mod} \, 2} h^{p,q}(f, \{ {\rm id} \})\\
&=& \sum_{q \in I(f,\{{\rm id}\})} h(q).
\end{eqnarray*}
By F.~Hirzebruch and D.~Zagier \cite[p.~105]{HZ}, this formula is equal to
\begin{eqnarray*}
{\rm sign}\, f & = & \sum_{q \in I(f,\{{\rm id}\})} \frac{2}{\pi \sqrt{-1}} \sum_{k=- \infty \atop k\, {\rm odd}}^\infty \frac{1}{k} e^{\pi \sqrt{-1} k q} \\
 & = & \frac{2}{\pi \sqrt{-1}} \sum_{k=- \infty \atop k\, {\rm odd}}^\infty \frac{1}{k} \left( \sum_{q \in I(f,\{{\rm id}\})} e^{\pi \sqrt{-1} k q} \right) \\
 & = & \frac{2}{\pi \sqrt{-1}} \sum_{k=- \infty \atop k\, {\rm odd}}^\infty \frac{1}{k} \prod_{i=1}^n \frac{e^{\pi \sqrt{-1} k w_i}-e^{\pi \sqrt{-1} k}}{1 - e^{\pi \sqrt{-1} k w_i}}.
\end{eqnarray*}
On the other hand, we have by the definition of the elliptic genus
\[ Z(\sqrt{-1} \infty, z) = \prod_{i=1}^n y^{-(1-2w_i)/2} \frac{(1-y^{1-w_i})}{(1-y^{w_i})} = y^{-n/2} \prod_{j=1}^n \frac{y^{w_i} -y}{1-y^{w_i}},
\]
where $y=e^{2 \pi \sqrt{-1} z}$.
From this we can deduce the formula of the proposition.
\end{proof}

\begin{remark} 
Since $Z(f,G)(\sqrt{-1} \tau,z) = y^{-\hat{c}/2} \chi(f,G)(y)$, the signature of $V_f$ can already be derived from the $\chi_y$-genus by substituting $y=e^{2 \pi \sqrt{-1} z}$.
\end{remark}

Proposition~\ref{prop:sign} leads to the following definition of an orbifoldized signature.
\begin{definition}
Let $f(x_1, \ldots , x_n)$ be a non-degenerate weighted homogeneous polynomial, $G$ be a subgroup of $G_f$, and let $n$ be odd. The {\em orbifoldized signature} is defined by
\[
{\rm sign}(f,G) := \left\{ \begin{array}{ll}  \frac{2}{\pi} \sum_{k=- \infty \atop k\, {\rm odd}}^\infty \frac{(-1)^{\frac{kn-1}{2}}}{k} Z(f,G)(\sqrt{-1} \infty, \frac{k}{2}) & \mbox{if } G \subset {\rm SL}(3; \CC) \cap G_f, \\
- \frac{2}{\pi} \sum_{k=- \infty \atop k\, {\rm odd}}^\infty \frac{(-1)^{\frac{kn-1}{2}}}{k} Z(f,G)(\sqrt{-1} \infty, \frac{k}{2}) & \mbox{if } G_0 \subset G.
\end{array} \right. 
\]
\end{definition}

\begin{remark} \label{rem:sign}
From the proof of Proposition~\ref{prop:sign}, one can see that the analogues of the right-hand side equalities of Theorem~\ref{thm:st:ourform} for $\{ {\rm id} \}$ replaced by $G$ hold.
\end{remark}

From Corollary~\ref{cor:ellgen} we get
\begin{corollary}
Let $f(x_1, \ldots , x_n)$ be an invertible polynomial, $G \subset {\rm SL}(3; \CC) \cap G_f$, and $n$ odd. Then 
\[
{\rm sign}(f,G) = {\rm sign}(\widetilde{f},\widetilde{G}).
\]
\end{corollary}

\section{Orbifold lattices}
Let $f(x_1, \ldots , x_n)$ be an invertible polynomial, $G$ a subgroup of $G_f$, and let $n=3$. We offer two lattices which conjecturally have the orbifoldized signature.

First we consider the case $G \subset {\rm SL}(3;\CC) \cap G_f$. 
The following construction is inspired by \cite{ET3}. We consider the crepant resolution $Y:=G\text{-}{\rm Hilb}(\CC^3)$  of $X:=\CC^3/G$. Denote by $W$ the subvariety of $X$ defined by the image of the smooth analytic subspace 
$V=V_f$ of $\CC^3$ under the natural projection $\pi:\CC^3\longrightarrow X$. Denote by $h: Y \to X$ the resolution map and let $Z:=h^{-1}(W)$. The notations are summarized in the following diagram:
\begin{equation*}
\xymatrix{
& Z \ar@{^{(}->}[r]\ar[d]_{h|_Z} & Y \ar[d]_h\\
V \ar[r]_{\pi|_V}  & W \ar@{^{(}->}[r] & X=\CC^3/G
}.
\end{equation*}
 We have an exact sequence \cite{ET3}
\begin{equation*}
\xymatrix{
0\ar[r] & H_3(Y,Z;\ZZ) \ar[r]^{\iota}  & H_2(Z;\ZZ) \ar[r] & H_2(Y;\ZZ) \ar[r] & H_2(Y,Z;\ZZ) \ar[r] & 0.
}
\end{equation*}
Moreover, we have
\[
H_4(Y;\ZZ) \cong H_4(Y,Z;\ZZ).
\]

\begin{definition} We define the {\em orbifoldized Milnor lattice} $\Lambda^{(A)}_{f,G}$ to be the free part of
\[
{\rm Im}( H_3(Y,Z;\ZZ) \stackrel{\iota}{\hookrightarrow} H_2(Z; \ZZ))
\]
equipped with the restriction of the intersection form on $H_2(Z; \ZZ)$ to the image of $\iota$.
\end{definition}

We have the following facts.

\begin{proposition} \label{prop:facts}
\begin{itemize}
\item[{\rm (i)}] If $G=\{ {\rm id} \}$, then $\Lambda^{(A)}_{f,G}$ is the usual Milnor lattice $H_2(V_f;\ZZ)$ with the intersection form.
\item[{\rm (ii)}] The rank of $\Lambda^{(A)}_{f,G}$ is
\[
{\rm rank}\, \Lambda^{(A)}_{f,G} = \sum_{p+q=3 \atop p,q \in \QQ} h^{p,q}(f,G).
\]
\item[{\rm (iii)}] For the invariants $(\mu_+, \mu_0, \mu_-)$ of the lattice $\Lambda^{(A)}_{f,G}$ we have
\[
(\mu_+, \mu_0, \mu_-) = \left( \sum_{p+q=3 \atop q < 1} 2h^{p,q}(f,G), \sum_{p+q=3 \atop q =1} 2h^{p,q}(f,G), \sum_{p+q=3 \atop 1 < q <2} h^{p,q}(f,G) \right).
\]
In particular, $\mu_+ - \mu_- = {\rm sign}(f,G)$.
\end{itemize}
\end{proposition}

\begin{proof}
(i) is trivial.

(ii) We have
\[
H^k(Y,Z;\ZZ)=0,\quad k \ne 2,3,4. 
\]
Moreover, one can show that there are the following isomorphisms  (cf.\ \cite[(2.16)]{ET3}):
\begin{eqnarray}
H_2(Y,Z;\ZZ) \otimes_\ZZ \QQ  & \cong & \QQ^{j_G}, \label{eq:2jG} \\
H_4(Y,Z;\ZZ) \otimes_\ZZ \QQ  & \cong & \QQ^{j_G}, \label{eq:4jG}
\end{eqnarray} 
where
\[
j_G := \sharp \{ g \in G \, | \, {\rm age}(g)=1, n_g=0 \}.
\]
Denote by $MHS^{1/d}$ the category of $\frac{1}{d} \ZZ$-indexed mixed Hodge structures (see \cite[Sect.~3.8]{Ya}). For $a \in \frac{1}{d} \ZZ$, let $\QQ(a)$ be the Tate-Hodge structure defined to be the one-dimensional $\frac{1}{d} \ZZ$-indexed Hodge structure $H$ such that $H^{-a,-a}$ is the only non-zero component of $H \otimes_{\QQ} \CC$. Denote the class $[\QQ(a)]$ of $\QQ(a)$ in $K_0(MHS^{1/d})$ by $\LL^{-a}$. By \cite[Corollary~72]{Ya}, we have the following equality for $Y$ in $K_0(MHS^{1/d})$
\begin{equation*}
\sum_i (-1)^i \left[ H^i_c(Y; \QQ) \right] = \sum_{g \in G} \LL^{n- {\rm age}(g)}.
\end{equation*}
By Poincar\'e duality, we get from this
\begin{equation}
\sum_i (-1)^i \left[ H^i(Y; \QQ) \right] = \sum_{g \in G} \LL^{{\rm age}(g)}. \label{eq:Y}
\end{equation}
Similarly, we obtain for $Z$
\begin{equation*}
\sum_i (-1)^i \left[ H^i_c(Z; \QQ) \right] = \sum_{g \in G} \left( \sum_i (-1)^i \left[ H^i_c([V^g/G]; \QQ) \right] \right) \LL^{n- n_g -{\rm age}(g)},
\end{equation*}
where $V^g:= \{ f^g=1 \} \subset \CC^{n_g}$. Again, it follows that
\begin{eqnarray}
\sum_i (-1)^i \left[ H^i(Z; \QQ) \right] &  = & \sum_{g \in G} \left( \sum_i (-1)^i \left[ H^i([V^g/G]; \QQ) \right] \right) \LL^{{\rm age}(g)} \nonumber \\
& = & \sum_{g \in G} \left( (-1)^{n_g} \left[ \widetilde{H}^{n_g-1}(V^g; \QQ)^G \right]  \LL^{{\rm age}(g)} + \LL^{{\rm age}(g)} \right). \label{eq:Z}
\end{eqnarray}
From Equations~(\ref{eq:Y}) and (\ref{eq:Z}) we obtain
\begin{equation}
\sum_i (-1)^i \left[ H^i(Y,Z; \QQ) \right] = \sum_{g \in G} (-1)^{n_g} \left[ \widetilde{H}^{n_g-1}(V^g; \QQ)^G \right]  \LL^{{\rm age}(g)}.
\end{equation}
In particular, we obtain
\[
\sum_i (-1)^i \dim_{\QQ} H^i(Y,Z; \QQ) = E(f,G)(1,1).
\]
By the isomorphisms (\ref{eq:2jG}) and (\ref{eq:4jG}), we get
\[ 
\dim_{\QQ} H^3(Y,Z; \QQ) = E(f,G)(1,1) -2j_G.
\]
More precisely, by comparing weights, we have the following isomorphisms of $\frac{1}{d} \ZZ$-indexed mixed Hodge structures:
\begin{eqnarray}
H^2(Y,Z; \QQ) & \cong & \QQ(-1)^{j_G} \quad (\mbox{weight 2}),  \\
H^4(Y,Z; \QQ) & \cong & \QQ(-2)^{j_G} \quad (\mbox{weight 4}),  \\
H^3(Y,Z; \QQ) & \cong &  H^2(V; \QQ)^G \oplus \bigoplus_{g \in G \atop {\rm age}(g)=1, n_g=1} \widetilde{H}^0(V^g;\QQ)^G(-1) \quad (\mbox{weight 3}). \label{eq:H3}
\end{eqnarray}
This implies the formula for the rank.

(iii) This follows from the isomorphism (\ref{eq:H3}) and Theorem~\ref{thm:st:ourform}, since the Tate twist $(-1)$ changes the sign of the intersection form. We also use Serre duality of Hodge numbers
\[
h^{n-p,n-q}(f,G)=h^{p,q}(f,G).
\]
The last statement of (iii) follows from Remark~\ref{rem:sign}.
\end{proof}

In \cite{ET1}, the pairs $(f,G)$ where $f$ defines an ADE singularity are classified. A.~Basalaev, E.~Werner and the second author \cite{BTW1} have shown that $Y$ can be covered by charts such that the  induced function $\widehat{f} : Y \to \CC$ has singularities only in one chart $Y_1 \cong \CC^3$. The same fact holds for the quotients of some of the exceptional unimodal singularities considered in \cite{BTW2} (see \cite[Sect.~6.4]{We}, \cite[Sect.~6]{BT}). Let the restriction of the function $\widehat{f}$ to this chart $Y_1$ be denoted by $\overline{f}$. Then we have the following proposition.
\begin{proposition}
Let $f$ define an ADE singularity or be one of the exceptional unimodal singularities in \cite[Table~3]{BTW2} and $G$ be a subgroup of ${\rm SL}(3;\CC) \cap G_f$. Then the lattice $\Lambda^{(A)}_{f,G}$ is isomorphic to the Milnor lattice of $\overline{f}$.
\end{proposition}

\begin{proof} Let $Z_1:=Z \cap Y_1$. Due to the long exact homology sequence and the Mayer--Vietoris sequence, there exists the following commutative diagram
\[ 
\xymatrix{
0 \ar[r] & H_3(Y,Z;\ZZ) \ar[r]^{\iota}  & H_2(Z; \ZZ)  \\
0 \ar[r] & H_3(Y_1,Z_1; \ZZ) \ar[r]^{\cong} \ar[u] & H_2(Z_1; \ZZ)  \ar[u]\\
 & 0 \ar[u] & 0 \ar[u]
 } 
\]
By \cite[Theorem~6.3.7]{BTW1} and \cite[Theorem~1]{BTW2}, we have
\[
E(f,G)(1,1) = E(\overline{f},\{ 1 \})(1,1), \quad j_G=0,
\]
which implies, by the proof of Proposition~\ref{prop:facts},
\[ {\rm rank}\, H_3(Y,Z;\ZZ) = {\rm rank} \, H_2(Z_1;\ZZ).
\]
Hence, the homomorphism
\[
H_2(Z_1;\ZZ) \to ({\rm Im}(H_3(Y,Z;\ZZ) \stackrel{\iota}{\hookrightarrow} H_2(Z;\ZZ)))_{\rm free}
\]
is an isomorphism, where, for an abelian group $H$,  $H_{\rm free}$ denotes the free part of the group $H$. Since the homomorphism $H_2(Z_1;\ZZ) \to H_2(Z; \ZZ)$ is compatible with the intersection forms on these groups, the lattice $\Lambda^{(A)}_{f,G}$ is isomorphic to the Milnor lattice of $\overline{f}$.
\end{proof}

Let $f(x_1,x_2,x_3) = \sum_{i=1}^3 c_i \prod_{j=1}^3 x_j^{E_{ij}}$ and let $W_f=(a_1,a_2,a_3;d)$ be the canonical system of weights of $f$, namely the unique solution of the equation
\begin{equation*}
E
\begin{pmatrix}
a_1\\
\vdots\\
a_n
\end{pmatrix}
={\rm det}(E)
\begin{pmatrix}
1\\
\vdots\\
1
\end{pmatrix}
,\quad 
d:={\rm det}(E).
\end{equation*}
Let 
\[
c_f:={\rm gcd}(a_1,a_2,a_3,d).
\]
By \cite[Proposition~3.1]{ET1}, we have $c_f=|\widetilde{G}_0|=|G_f/G_0|$.
The number
\[
a_{W_f} := d- a_1-a_2-a_3
\]
is called the {\em Gorenstein parameter} of $W_f$.  The reduced system of weights 
\[
W_f^{\rm red}=\left(\frac{a_1}{c_f},\frac{a_2}{c_f},\frac{a_3}{c_f};\frac{d}{c_f}\right)
\]
is a regular system of weights in the sense of K.~Saito \cite{Sa1,Sa2}. Then the usual Gorenstein parameter is 
\[
\eps_{W_f} := - \frac{a_{W_f}}{c_f}.
\]

There are 14+8 pairs $(f,G)$ of invertible polynomials $f$ giving regular systems of weights with $\eps_{W_f}=-1$ and subgroups $G \subset G_f$  with $G_0 \subset G$ such that $j_{\widetilde{G}}=0$ (cf., e.g., \cite{KST2}). The first 14 cases define the 14 exceptional unimodal singularities. For these systems of weights we can find invertible polynomials $f$ with $G_0=G_f$. The Berglund-H\"ubsch duality $f \leftrightarrow \widetilde{f}$ is just Arnold's strange duality and $\widetilde{G}_0=\{ {\rm id} \}$. For the remaining 8 cases we have indicated invertible polynomials and their duals in Table~\ref{Tab8}. We have also indicated the names of the singularities in Arnold's classification and the groups $\widetilde{G}_0 \simeq  G_f/G_0$.
\begin{table}[h]
\begin{center}
\begin{tabular}{cccccc}
\hline
$a_1,a_2,a_3;d$ & Name & $f$   & $\widetilde{f}$ & $G_f/G_0$ & $\alpha_1, \ldots ,\alpha_m$ \\
\hline
$2,6,9;18$ & $J_{3,0}$ &  $x^6y+y^3+z^2$ & $x^6+xy^3+z^2$ & $\ZZ/2\ZZ$  & $2,2,2,3$ \\
$2,4,7;14$ & $Z_{1,0}$  & $x^5y + xy^3 +z^2$ & $x^5y +xy^3 + z^2$ & $\ZZ/2\ZZ$ & $2,2,2,4$ \\
$2,4,5;12$ & $Q_{2,0}$ & $x^4y + y^3 + xz^2$ & $x^4z+xy^3+z^2$  & $\ZZ/2\ZZ$ & $2,2,2,5$\\
$2,3,6;12$ & $W_{1,0}$ & $x^6+y^2+yz^2$  & $x^6+y^2z+z^2$ & $\ZZ/2\ZZ$ & $2,2,3,3$ \\
$2,3,4;10$ & $S_{1,0}$  & $x^5+xy^2+yz^2$ & $x^5y+y^2z+z^2$ & $\ZZ/2\ZZ$ & $2,2,3,4$\\
$2,3,3;9$  & $U_{1,0}$ & $x^3+xy^2+yz^3$  &  $x^3y+y^2z+z^3$ & $\ZZ/2\ZZ$ & $2,3,3,3$  \\
$2,2,5;10$ & $NA^1_{0,0}$ & $x^5+y^5+z^2$ & $x^5+y^5+z^2$ & $\ZZ/5\ZZ$ & $2,2,2,2,2$ \\
$2,2,3;8$ & $VNA^1_{0,0}$ & $x^4+y^4+yz^2$ & $x^4+y^4z+z^2$ & $\ZZ/4\ZZ$ & $2,2,2,2,3$ \\
\hline
\end{tabular}
\end{center}
\caption{Regular systems of weights with $\eps_W=-1$ and $G_0 \neq G_f$} \label{Tab8}
\end{table}

\begin{figure}
$$
\xymatrix{ & & & {\cdots} & & & \\
& & & *{\bullet} \ar@{-}[d] \ar@{}_{\delta_2}[l] & & & \\
 & & & *{\bullet}  \ar@{==}[d] \ar@{-}[dr]  \ar@{-}[ldd] \ar@{-}[rdd] \ar@{}_{\delta_1}[l]
 & & &  \\
 *{\bullet} \ar@{-}[r] \ar@{}_{\delta^2_1}[d]  & {\cdots} \ar@{-}[r]  & *{\bullet} \ar@{-}[r] \ar@{-}[ur]   \ar@{}_{\delta^2_{\alpha_2-1}}[d] & *{\bullet} \ar@{-}[dl]  \ar@{-}[dr] \ar@{-}[r]  \ar@{}^{\delta_{0}}[l]  & *{\bullet} \ar@{-}[r]  \ar@{}^{\delta^{m-1}_{\alpha_{m-1}-1}}[d]  & {\cdots} \ar@{-}[r]  &*{\bullet} \ar@{}^{\delta^{m-1}_1}[d]   \\
& &  *{\bullet} \ar@{-}[dl] \ar@{}_{\delta^1_{\alpha_1-1}}[r]  &  & *{\bullet} \ar@{-}[dr] \ar@{}^{\delta^m_{\alpha_m-1}}[l] &  &  \\
 & {\cdots} \ar@{-}[dl] & &  & & {\cdots} \ar@{-}[dr] & \\
*{\bullet}  \ar@{}_{\delta^1_1}[r] & & &  & & & *{\bullet}  
\ar@{}^{\delta^r_1}[l]
  } 
$$
\caption{The graph $T^+_{\alpha_1,\ldots ,\alpha_m}$} \label{FigT+}
\end{figure}
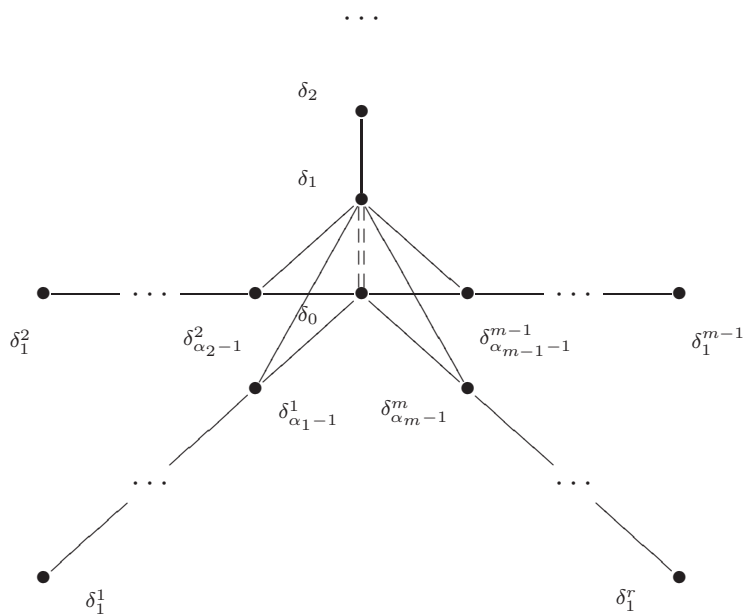

\begin{proposition} \label{prop:8}
For each of the Landau--Ginzburg orbifolds $(\widetilde{f}, \widetilde{G}_0)$ of Table~\ref{Tab8}, there exists a basis of $\Lambda^{(A)}_{\widetilde{f}, \widetilde{G}_0} \otimes \QQ$ with a Coxeter--Dynkin diagram in the shape of Fig.~\ref{FigT+}, where the numbers $\alpha_1, \ldots , \alpha_m$ are given in Table~\ref{Tab8}.
\end{proposition}

\begin{remark} The numbers $\alpha_1, \ldots , \alpha_m$ are the Dolgachev numbers of $f$ which correspond to the Gabrielov numbers of $\widetilde{f}$, see \cite{ET1}.
\end{remark}

\begin{proof}[Proof of Proposition~\ref{prop:8}]
To construct a basis of $\Lambda^{(A)}_{\widetilde{f}, \widetilde{G}_0} \otimes \QQ$ with the required Coxeter--Dynkin diagram, we shall follow the method of \cite{ET3}. We first construct a basis of $H_2(W; \QQ)$ as follows. Consider a pair $(\widetilde{f}, \widetilde{G}_0)$ from Table~\ref{Tab8}. We have chosen the coordinates in such a way that the group $\widetilde{G}_0$ acts on the coordinates $x,y$ and leaves the coordinate $z$ invariant. For $i=1,2,3$, let $K_i$ be the maximal subgroup of $\widetilde{G}_0$ fixing the $i$-th coordinate $x_i$. Denote the order $|K_i|$ of $K_i$ by $n_i$. Then we have $n_1=n_2=1$. Let $\gamma_1', \gamma_2', \gamma_3'$ be the Gabrielov numbers of the pair $(\widetilde{f}, \{ {\rm id} \})$ defined in \cite{ET0}. For $i=1,2,3$, set
\begin{equation*}
\gamma_i:=\frac{\gamma'_i}{\left|G/K_i\right|}.
\end{equation*}
Then $\gamma_1, \gamma_2$ and $\gamma_3$ repeated $n_3$ times are the Gabrielov numbers $\Gamma_{(\widetilde{f}, \widetilde{G}_0)}$ of the pair $(\widetilde{f}, \widetilde{G}_0)$ (see \cite{ET1}). (Note that we omit numbers which are equal to one in $\Gamma_{(\widetilde{f}, \widetilde{G}_0)}$.)
The values of $\gamma_1', \gamma_2', \gamma_3'$, $\gamma_1, \gamma_2, \gamma_3$, and $n_3$ for the pairs $(\widetilde{f}, \widetilde{G}_0)$ are indicated in Table~\ref{TabGab}. Then we shall show that there exists a basis
\[
\overline{\mathcal B}= ( \overline{\delta}_0, \overline{\delta}_1, \overline{\delta}_2; \overline{\delta}^1_1, \ldots, \overline{\delta}^1_{\gamma_1-1}; \overline{\delta}^2_1, \ldots, \overline{\delta}^2_{\gamma_2-1}; \overline{\delta}_1^3, \ldots, \overline{\delta}_{\gamma_3-1}^3 )
\]
of $H_2(W; \QQ)$ with the Coxeter--Dynkin diagram indicated in Fig.~\ref{FigbarTpqr}. (Here, if $\gamma_i=1$, then there is no branch corresponding to $i$.) As, in \cite{ET3}, one can construct from this basis a basis of $\Lambda^{(A)}_{\widetilde{f}, \widetilde{G}_0} \otimes \QQ$ with a Coxeter--Dynkin diagram in the shape of Fig.~\ref{FigT+}.
\begin{table}[h]
\begin{center}
\begin{tabular}{cccccc}
\hline
 Name  & $\widetilde{f}$ & $\gamma_1', \gamma_2', \gamma_3'$ & $\gamma_1, \gamma_2, \gamma_3$ & $n_3$ & $\Gamma_{(\widetilde{f}, \widetilde{G}_0)}$\\
\hline
 $J_{3,0}$  & $x^6+xy^3+z^2$ & $4,6,2$  & $2,3,2$ & 2 & $2,3,2,2$\\
 $Z_{1,0}$  & $x^5y +xy^3 + z^2$ & $4,8,2$ & $2,4,2$ & 2 & $2,4,2,2$\\
 $Q_{2,0}$  & $x^4z+xy^3+z^2$  & $4,10,2$ & $2,5,2$ & 2 & $2,5,2,2$\\
 $W_{1,0}$   & $x^6+y^2z+z^2$ & $6,6,2$ & $3,3,2$ & 2 & $3,3,2,2$ \\
 $S_{1,0}$   & $x^5y+y^2z+z^2$ & $6,8,2$ & $3,4,2$ & 2 & $3,4,2,2$ \\
 $U_{1,0}$   &  $x^3y+y^2z+z^3$ & $4,6,3$ & $2,3,3$ & 2 & $2,3,3,3$ \\
 $NA^1_{0,0}$  & $x^5+y^5+z^2$ & $5,5,2$ & $1,1,2$ & 5 & $2,2,2,2,2$\\
 $VNA^1_{0,0}$  & $x^4+y^4z+z^2$ & $4,12,2$ & $1,3,2$ & 4 & $3,2,2,2,2$\\
\hline
\end{tabular}
\end{center}
\caption{Gabrielov numbers} \label{TabGab}
\end{table}

\begin{figure}
$$
\xymatrix{ 
& & & *{\bullet} \ar@{-}[d] \ar@{}_{\overline{\delta}_2}[l] & & & \\
 & & & *{\bullet}  \ar@{==}[d] \ar@{-}[dr]^{n_3}  \ar@{-}[ldd]  \ar@{}_{\overline{\delta}_1}[l]
 & & &  \\
*{\bullet} \ar@{-}[r] \ar@{}_{\overline{\delta}^2_{\gamma_2-1}}[d]  & {\cdots} \ar@{-}[r]  & *{\bullet} \ar@{-}[r] \ar@{-}[ur]    \ar@{}_{\overline{\delta}^2_1}[d] & *{\bullet} \ar@{-}[dl] \ar@{-}[r]^{n_3}  \ar@{}^{\overline{\delta}_{0}}[l]  &  *+=[o][F-:<3pt>]\txt{\ $\scriptstyle{-2n_3}$\,\, } \ar@{-}[r]^{n_3}  \ar@{}^{\overline{\delta}_1^3}[d]  & {\cdots} \ar@{-}[r]^{n_3}  & *+=[o][F-:<3pt>]\txt{\ $\scriptstyle{-2n_3}$\,\,} \ar@{}^{\overline{\delta}_{\gamma_3-1}^3}[d]   \\
& &  *{\bullet} \ar@{-}[dl] \ar@{}_{\overline{\delta}^1_1}[r]   & & & &  \\
 & {\cdots} \ar@{-}[dl] & & & & & \\
 *{\bullet}   \ar@{}_{\overline{\delta}^1_{\gamma_1-1}}[r]& & & & & &
  }
$$
\caption{Coxeter--Dynkin diagram corresponding to $\overline{\mathcal B}$}  \label{FigbarTpqr}
\end{figure}
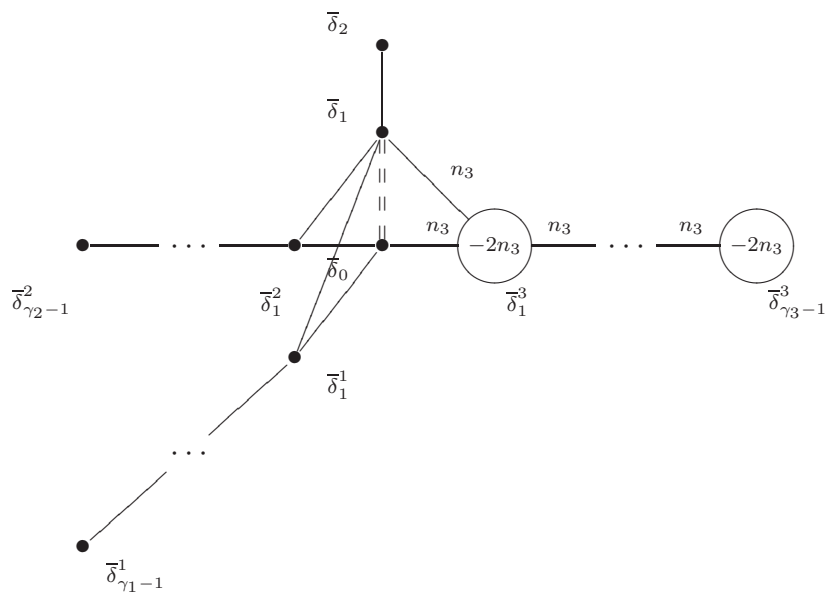

We now indicate for each case a basis $\overline{\mathcal B}$. All the singularities except $U_{1,0}$ can be given by equations of the form $h(x,y)+z^2=0$. Therefore it suffices in these cases to indicate a Coxeter--Dynkin diagram for a distinguished basis of vanishing cycles for the plane curve singularity defined by $h(x,y)$.

$J_{3,0}$, $Z_{1,0}$, $Q_{2,0}$, $W_{1,0}$, $S_{1,0}$: Here we only treat the case $J_{3,0}$. Coxeter--Dynkin diagrams for the other cases can be given by extensions of the Coxeter--Dynkin diagram for $J_{3,0}$. (See \cite[Sect.~8, Example~2]{EG} for the case $Z_{1,0}$.) A Coxeter--Dynkin diagram with respect to a suitable real morsification of the function $h$ due to the method of N.~A'Campo and Gusein-Zade \cite{A'C,GZ1,GZ2} is given in Fig.~\ref{FigA}.
\begin{figure}
$$
\xymatrix{   & & & & & & \\
   &  *{\bullet} \ar@{-}[r] \ar@{-}[d]  \ar@{--}[dr] \ar@{}^{\delta_6}[u] & *{\bullet} \ar@{-}[d] \ar@{-}[r] \ar@{}^{\delta_5}[u] & *{\bullet} \ar@{-}[r]  \ar@{-}[d] \ar@{--}[dl] \ar@{--}[dr] \ar@{}^{\delta_2}[u]  &*{\bullet} \ar@{-}[d] \ar@{}^{\delta_{11}}[u]  & & \\
 & *{\bullet} \ar@{-}[r] \ar@{}_{\delta_3}[d] &   *{\bullet} \ar@{-}[r] \ar@{-}[r] \ar@{-}[d] \ar@{--}[dr] \ar@{}_{\delta_1}[d]  & *{\bullet} \ar@{-}[r] \ar@{-}[d] \ar@{}_{\delta_7}[d]  & *{\bullet} \ar@{-}[r] \ar@{-}[d] \ar@{--}[dl]\ar@{--}[dr] \ar@{}_{\delta_8}[d] & *{\bullet} \ar@{-}[d] \ar@{}_{\delta_{10}}[d] &  \\
  &  & *{\bullet} \ar@{-}[r] \ar@{}_{\delta_4}[d] & *{\bullet} \ar@{-}[r] \ar@{}_{\delta_9}[d] & *{\bullet} \ar@{-}[r] \ar@{}_{\delta_{12}}[d]  & *{\bullet} \ar@{}_{\delta_{13}}[d] & \\
  & & & & & & } 
$$
\caption{Coxeter--Dynkin diagram of $h(x,y)=x^6+xy^3$} \label{FigA}
\end{figure}
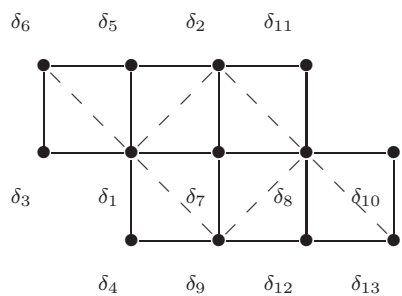
The group $\widetilde{G}_0$ acts by reflection at the central vertex $\delta_7$. Let
\[
\overline{\delta}'_i = (\pi|_V)_\ast(\delta_i) = (\pi|_V)_\ast(\delta_{i+7}) \quad \mbox{for } i=1, \ldots, 6, \quad \overline{\delta}'_7= (\pi|_V)_\ast(\delta_7),
\]
and let
\begin{eqnarray*}
\overline{\delta}_3 & = &  \overline{\delta}'_3 + \overline{\delta}'_1 +\overline{\delta}'_2+ \overline{\delta}'_4 +\overline{\delta}'_5 +\overline{\delta}'_7, \\
\overline{\delta}_6 & = & \overline{\delta}'_6+ \overline{\delta}'_5, \\
\overline{\delta}_i & = & \overline{\delta}'_i \quad \mbox{for }i=1,2,4,5,7.
\end{eqnarray*}
Then $\overline{\mathcal B} = ( \overline{\delta}_1, \ldots, \overline{\delta}_7)$ is a basis intersecting according to the graph of Fig.~\ref{FigbarTpqr}.

$NA^1_{0,0}$: Here the function $h$ is of the form $h(x,y)=x^5 + y^5$. A distinguished basis of vanishing cycles for this polynomial can be computed by the method of A.~M.~Gabrielov \cite{Ga1}. It is obtained as follows: Let $e_1,e_2, e_3, e_4$ be a distinguished basis of vanishing cycles for $A_4$ ($x^5$) and $f_1, f_2, f_3, f_4$ a distinguished basis of vanishing cycles for 
$A_4$ ($y^5$). Then 
\begin{equation} \label{eq:Gab}
\delta_{ij} = e_i \otimes f_j
\end{equation}
is a distinguished basis of vanishing cycles for $h$. We extend these sets by $e_5:=-(e_1+e_2+e_3+e_4)$ and $f_5:=-(f_1 + f_2+f_3+f_4)$. We extend the definition (\ref{eq:Gab}) to $i=5$ and $j=5$ as well. One has
\[ \langle \delta_{ij}, \delta_{ij} \rangle =-2, \quad \langle \delta_{ij}, \delta_{i+1,j} \rangle =\langle \delta_{ij}, \delta_{i,j+1} \rangle =1, \quad \langle \delta_{ij}, \delta_{i+1,j+1} \rangle=-1
\]
and $\langle \delta_{ij}, \delta_{i'j'} \rangle =0$ otherwise, where $i+1=1$ for $i=5$ and $j+1=1$ for $j=5$ (cf. \cite[Sect.~8, Example~3]{EG}). The group $\widetilde{G}_0$ is generated by the matrix $g={\rm diag}(\zeta,\zeta^{-1})$ where $\zeta$ is a primitive 5-th root of unity. Then the generator $g$ acts as follows:
\[ g(e_i)=e_{i+1} \ \mbox{mod} \, 5, \quad g(f_i)=f_{i-1} \ \mbox{mod} \, 5.
\]
Let 
\begin{eqnarray*}
\overline{\delta}'_1 & = & (\pi|_V)_\ast(\delta_{11})=(\pi|_V)_\ast(\delta_{25}) =(\pi|_V)_\ast(\delta_{34})=(\pi|_V)_\ast(\delta_{43})=(\pi|_V)_\ast(\delta_{52}), \\
\overline{\delta}'_2 & = & (\pi|_V)_\ast(\delta_{12})= (\pi|_V)_\ast(\delta_{21})= (\pi|_V)_\ast(\delta_{35})= (\pi|_V)_\ast(\delta_{44})= (\pi|_V)_\ast(\delta_{53}), \\
\overline{\delta}'_3  & = & (\pi|_V)_\ast(\delta_{13})=(\pi|_V)_\ast(\delta_{22})=(\pi|_V)_\ast(\delta_{31})=(\pi|_V)_\ast(\delta_{45})=(\pi|_V)_\ast(\delta_{54}), \\
\overline{\delta}'_4 & = & (\pi|_V)_\ast(\delta_{14})=(\pi|_V)_\ast(\delta_{23})=(\pi|_V)_\ast(\delta_{32})=(\pi|_V)_\ast(\delta_{41})=(\pi|_V)_\ast(\delta_{55}).
\end{eqnarray*}
Using the Sage Mathematics software \cite{Sage}, one can compute that the following basis transformation
\[ \left( \begin{array}{c} \overline{\delta}_0 \\ \overline{\delta}_1 \\ \overline{\delta}_2 \\ \overline{\delta}^3_1 \end{array} \right) = 
\left( \begin{array}{cccc}  0 & 1 & 2 & 0 \\-1 & 0 & 1 & 0 \\ 1 & 0 & 0 & 0 \\ -1 & -2 & -3 & 1 \end{array} \right) \left( \begin{array}{c} \overline{\delta}'_1 \\ \overline{\delta}'_2 \\\overline{\delta}'_3 \\ \overline{\delta}'_4 \end{array} \right)
\]
gives a basis with the required Coxeter--Dynkin diagram.

$VNA^1_{0,0}$: By a change of coordinates, $\widetilde{f}$ can be given by $h(x,y)+z^2$ with $h(x,y)=x^4+y^8$. By the same method as in the case $NA^1_{0,0}$, we obtain a basis
\begin{eqnarray*}
\overline{\delta}'_1 & = & (\pi|_V)_\ast(\delta_{11})=(\pi|_V)_\ast(\delta_{27}) =(\pi|_V)_\ast(\delta_{35})=\pi|_V)_\ast(\delta_{43}) \\
\overline{\delta}'_2 & = & (\pi|_V)_\ast(\delta_{12})= (\pi|_V)_\ast(\delta_{28})= (\pi|_V)_\ast(\delta_{36})= (\pi|_V)_\ast(\delta_{44}) \\
\overline{\delta}'_3  & = & (\pi|_V)_\ast(\delta_{14})=(\pi|_V)_\ast(\delta_{22})=(\pi|_V)_\ast(\delta_{38})=(\pi|_V)_\ast(\delta_{46}) \\
\overline{\delta}'_4 & = & (\pi|_V)_\ast(\delta_{15})=(\pi|_V)_\ast(\delta_{23})=(\pi|_V)_\ast(\delta_{31})=(\pi|_V)_\ast(\delta_{47}) \\
\overline{\delta}'_5 & = & (\pi|_V)_\ast(\delta_{17})=(\pi|_V)_\ast(\delta_{25})=(\pi|_V)_\ast(\delta_{33})=(\pi|_V)_\ast(\delta_{41}) \\
\overline{\delta}'_6 & = & (\pi|_V)_\ast(\delta_{18})=(\pi|_V)_\ast(\delta_{26})=(\pi|_V)_\ast(\delta_{34})=(\pi|_V)_\ast(\delta_{42}) \\
\end{eqnarray*}
Again using Sage, one can compute that the following basis transformation
\[ 
\left( \begin{array}{c} \overline{\delta}_0 \\ \overline{\delta}_1 \\ \overline{\delta}_2 \\ \overline{\delta}^2_1 \\ \overline{\delta}^2_2 \\ \overline{\delta}^3_1 \end{array} \right) = 
\left( \begin{array}{cccccc}  
0 & -3 & 0 & 3 & 2 & 0 \\
0 & 0 & 1 & 0 & 0 & 0\\
0 & -2 & -1 & 2 & 1 &0  \\ 
-1 & 0 & 0 & 0 & 0 & 0 \\
1 & 1 & 0 & 0 & 0 & 0 \\
3 & 10 & -1 & -11 & -6 & 1 \end{array} \right) 
\left( \begin{array}{c} \overline{\delta}'_1 \\ \overline{\delta}'_2 \\\overline{\delta}'_3 \\ \overline{\delta}'_4 \\ \overline{\delta}'_5 \\ \overline{\delta}'_6 \end{array} \right)
\]
gives a basis with the required Coxeter--Dynkin diagram.

$U_{1,0}$: Here we use a Coxeter--Dynkin diagram which is obtained by the method of \cite[Theorem~1]{Ga2}. This is the graph depicted in Fig.~\ref{FigU10}.
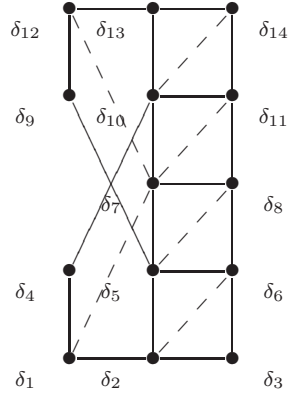
\begin{figure}
$$
\xymatrix{  
   &  *{\bullet} \ar@{-}[r]  \ar@{-}[d] \ar@{--}[ddr] \ar@{}^{\delta_{12}}[l] & *{\bullet} \ar@{-}[d] \ar@{-}[r] \ar@{}^{\delta_{13}}[l] & *{\bullet}  \ar@{-}[d] \ar@{--}[dl]  \ar@{}_{\delta_{14}}[r]  & \\
 & *{\bullet} \ar@{-}[ddr]  \ar@{}^{\delta_9}[l] &   *{\bullet} \ar@{-}[r]  \ar@{-}[d] \ar@{}^{\delta_{10}}[l]  & *{\bullet}  \ar@{-}[d] \ar@{--}[dl] \ar@{}_{\delta_{11}}[r]  &  \\
  & & *{\bullet} \ar@{-}[r] \ar@{-}[d] \ar@{}^{\delta_7}[l] & *{\bullet} \ar@{-}[d] \ar@{--}[dl]  \ar@{}_{\delta_8}[r]  &  \\
    &  *{\bullet} \ar@{-}[d] \ar@{-}[uur]   \ar@{}^{\delta_4}[l] & *{\bullet} \ar@{-}[d] \ar@{-}[r] \ar@{}^{\delta_5}[l] & *{\bullet}  \ar@{-}[d] \ar@{--}[dl]  \ar@{}_{\delta_6}[r]  & \\
    &  *{\bullet}  \ar@{-}[r]  \ar@{--}[uur]   \ar@{}^{\delta_1}[l] & *{\bullet}  \ar@{-}[r] \ar@{}^{\delta_2}[l] & *{\bullet}   \ar@{}_{\delta_3}[r]  &  } 
$$
\caption{Coxeter--Dynkin diagram of $U_{1,0}$} \label{FigU10}
\end{figure}
The action of the group $\widetilde{G}_0 \cong \ZZ/2\ZZ$ on the basis $(\delta_1, \ldots, \delta_{14})$ is seen from the following formulas:
\[
\begin{array}{ll} 
\overline{\delta}'_1= (\pi|_V)_\ast(\delta_4) = (\pi|_V)_\ast(\delta_9), & \overline{\delta}'_2= (\pi|_V)_\ast(\delta_1) = (\pi|_V)_\ast(\delta_{12}), \\
\overline{\delta}'_3= (\pi|_V)_\ast(\delta_2+\delta_3)=(\pi|_V)_\ast(\delta_{13}), & \overline{\delta}'_4= (\pi|_V)_\ast(\delta_2) = (\pi|_V)_\ast(\delta_{13}+\delta_{14}), \\
\overline{\delta}'_5=(\pi|_V)_\ast(\delta_5+\delta_6)= (\pi|_V)_\ast(\delta_{10}), & \overline{\delta}'_6= (\pi|_V)_\ast(\delta_5)=(\pi|_V)_\ast(\delta_{10}+\delta_{11}), \\
\overline{\delta}'_7= (\pi|_V)_\ast(\delta_7), & \overline{\delta}'_8= (\pi|_V)_\ast(\delta_8).
\end{array}
\]
Using Sage, one obtains the following basis transformation:
\[ 
\left( \begin{array}{c} \overline{\delta}_0 \\ \overline{\delta}_1 \\ \overline{\delta}_2 \\ \overline{\delta}^1_1 \\ \overline{\delta}^1_2 \\ \overline{\delta}^2_1 \\  \overline{\delta}^3_1 \\ \overline{\delta}^3_2 \end{array} \right) = 
\left( \begin{array}{cccccccc}  
-1 & 0 & 0 & -1 & 0 & -1 & 0 & 0\\
0  & 1 & 0 & 0 & 0 & 0 & 0 & 0\\
0 & 0 & 1 & 0 & 0 & 0 & 0 & 0\\ 
0 & -1 & 0 & -1 & 0 & 0 & 0 & 0\\
0 & 0 & 0 & 0 & 1 & -1 & 0 & 0\\
0 & 0 & 0 & 1 & 0 & 1 & 0 & 0 \\
2 & 0 & 0 & 2 & 0 & 2 & 1 & 0\\
2 & 2 & 2 & 0 & 2 & 0 & -1 & 1
\end{array} \right) 
\left( \begin{array}{c} \overline{\delta}'_1 \\ \overline{\delta}'_2 \\\overline{\delta}'_3 \\ \overline{\delta}'_4 \\ \overline{\delta}'_5 \\ \overline{\delta}'_6 \\ \overline{\delta}'_7 \\ \overline{\delta}'_8
\end{array} \right)
\]

\end{proof}

On the other hand, we consider the case $G_0 \subset G$. 

Let $f(x_1,\dots ,x_n)=\sum_{i=1}^nc_i\prod_{j=1}^nx_j^{E_{ij}}$ be an invertible polynomial. Consider the free abelian group $\oplus_{i=1}^n\ZZ\vec{x_i}\oplus \ZZ\vec{f}$ 
generated by the symbols $\vec{x_i}$ for the variables $x_i$ for $i=1,\dots, n$
and the symbol $\vec{f}$ for the polynomial $f$.
The {\em maximal grading} $L_f$ of the invertible polynomial $f$ 
is the abelian group defined by the quotient 
\[
L_f:=\bigoplus_{i=1}^n\ZZ\vec{x_i}\oplus \ZZ\vec{f}\left/I_f\right.,
\]
where $I_f$ is the subgroup generated by the elements 
\[
\vec{f}-\sum_{j=1}^nE_{ij}\vec{x_j},\quad i=1,\dots ,n.
\]
The {\em maximal abelian symmetry group} $\widehat{G}_f$ of $f$ is the abelian group defined by 
\[
\widehat{G}_f:={\rm Spec}(\CC L_f)(\CC).
\]
For a subgroup  $G_0 \subset G \subset G_f$, let $\widehat{G}$ be the subgroup of $\widehat{G}_f$ 
defined by the following commutative diagram of short exact sequences
\[ 
\xymatrix{
\{ 1 \}\ar[r] & G \ar[r]\ar@{^{(}->}[d]  & \widehat{G} \ar[r]\ar@{^{(}->}[d] &  \CC^\ast \ar[r]\ar@{=}[d]& \{ 1 \}\\
\{ 1 \}\ar[r] & G_f \ar[r] & \widehat{G}_f \ar[r] & \CC^\ast \ar[r]& \{ 1 \}
}.
\]
Define $L_{\widehat{G}}$ as the quotient of the abelian group $L_f$ such that $\widehat{G}={\rm Spec}(\CC L_{\widehat{G}})(\CC)$.

Let ${\rm HMF}^{L_{\widehat{G}}}(f)$ be the stable homotopy category of $L_{\widehat{G}}$-graded matrix factorisations of $f$. Let $K_0({\rm HMF}^{L_{\widehat{G}}}(f))$ be the Grothendieck K-group of the category which is equipped with the Euler form $\chi$. Let $N({\rm HMF}^{L_{\widehat{G}}}(f))$ be the numerical Grothendieck group which is obtained from $K_0({\rm HMF}^{L_{\widehat{G}}}(f))$ by dividing out the radical of the Euler form.

\begin{definition} We define
\[
\Lambda^{(B)}_{f,G} := N({\rm HMF}^{L_{\widehat{G}}}(f))
\]
equipped with the symmetric bilinear form $(-1)^{\frac{n(n-1)}{2}}(\chi + \chi^T)$.
\end{definition}

We now restrict to the case $n=3$.

\begin{theorem} \label{thm:=ranks}
For $n=3$, we have
\[ 
{\rm rank}\, \Lambda^{(B)}_{f,G} = {\rm rank}\, \Lambda^{(A)}_{\widetilde{f},\widetilde{G}}.
\]
\end{theorem}

\begin{proof} 
Since $G_0 \subset G$ we have $\widetilde{G} \subset {\rm SL}(3;\CC) \cap G_{\widetilde{f}}$. For $i=1,2,3$, let $K_i$ be the maximal subgroup of $\widetilde{G}$ fixing the $i$-th coordinate $x_i$, whose order $|K_i|$ is denoted by $n_i$. By \cite[Corollary~2]{ET3} we have
\begin{equation} \label{eq:|G|}
\left|\widetilde{G}\right|=1+2j_{\widetilde{G}}+\sum_{i=1}^3\left(n_i-1\right),
\end{equation}
where $j_{\widetilde{G}}$ is the number of elements  $g \in \widetilde{G}$ such that ${\rm age}(g)=1$ and $N_g=0$. For simplicity of notation, we denote the rational weights of $\widetilde{f}$ by $w_1, w_2,w_3$.

By Proposition~\ref{prop:facts} and the definition of the numbers $h^{p,q}(\widetilde{f},\widetilde{G})$ we have
\begin{eqnarray*}
{\rm rank}\, \Lambda^{(A)}_{\widetilde{f},\widetilde{G}}  & = & \sum_{p+q=3 \atop p,q \in \QQ} h^{p,q}(\widetilde{f},\widetilde{G})\\
 & = & \frac{1}{|\widetilde{G}|} \prod_{i=1}^3 \left( \frac{1}{w_i} -1\right) + \frac{1}{|\widetilde{G}|} \sum_{i=1}^3 \left( \frac{1}{w_i} -1\right)(n_i -1) - 2\frac{1}{|\widetilde{G}|}j_{\widetilde{G}} \\
 &  & {} + \frac{1}{|\widetilde{G}|} \sum_{i=1}^3 (n_i-1)\left( \frac{1}{w_i} -1\right)(n_i -1) - \frac{1}{|\widetilde{G}|}\sum_{i=1}^3 (n_i-1)(|\widetilde{G}|-n_i) 
\end{eqnarray*}
By Equation~(\ref{eq:|G|}) we obtain
\begin{eqnarray*}
\lefteqn{{\rm rank}\, \Lambda^{(A)}_{\widetilde{f},\widetilde{G}}} \\
 & = & \frac{1}{|\widetilde{G}|} \prod_{i=1}^3 \left( \frac{1}{w_i} -1\right) + \frac{1}{|\widetilde{G}|} \sum_{i=1}^3 \left( \frac{1}{w_i} -1\right)(n_i -1) - \frac{1}{|\widetilde{G}|}\left(|\widetilde{G}|-1-\sum_{i=1}^3(n_i-1) \right) \\
 & & {} + \frac{1}{|\widetilde{G}|} \sum_{i=1}^3 \left( \frac{1}{w_i}-1\right)(n_i-1)n_i - \sum_{i=1}^3(n_i-1) + \frac{1}{|\widetilde{G}|}\sum_{i=1}^3(n_i-1)n_i.
\end{eqnarray*}
This equation can be simplified to
\begin{eqnarray} 
\lefteqn{{\rm rank}\, \Lambda^{(A)}_{\widetilde{f},\widetilde{G}}} \nonumber \\
 & = & \frac{1}{|\widetilde{G}|} \prod_{i=1}^3 \left( \frac{1}{w_i} -1\right) + \frac{1}{|\widetilde{G}|} \sum_{i=1}^3 \frac{1}{w_i} (n_i^2-1) + \frac{1}{|\widetilde{G}|} +2-\sum_{i=1}^3 n_i \nonumber \\ 
 & = & \frac{1}{|\widetilde{G}|} \left( \frac{1}{w_1w_2w_3} - \frac{1}{w_2w_3} - \frac{1}{w_1w_3} - \frac{1}{w_1w_2} \right) +2+ \sum_{i=1}^3 \left( \frac{1}{w_i} \frac{n_i}{|\widetilde{G}|} - 1 \right)n_i \label{eq:rank}
\end{eqnarray} 
By a case-by-case verification using the classification of $\widetilde{f}$ (see \cite{ET0}) and the $\widetilde{G}$-action on $\widetilde{f}$, one can show that Equation~(\ref{eq:rank}) becomes
\begin{eqnarray*}
{\rm rank}\, \Lambda^{(A)}_{\widetilde{f},\widetilde{G}} & = & 2+ \sum_{i=1}^3 \left( \frac{\widetilde{\gamma'}_i}{|\widetilde{G}/K_i|} -1 \right) \ast |K_i| + \frac{1}{|\widetilde{G}|} a_{W_{\widetilde{f}}},
\end{eqnarray*}
where $\widetilde{\gamma}'_i$, $i=1,2,3$, are the Gabrielov numbers of the pair $(\widetilde{f}, \{ {\rm id} \})$ and by $u \ast v$,  for positive integers $u$ and $v$, we denote $v$ copies of the integer $u$. (Here we use $n_3=1$ for type~II and $n_2=n_3=1$ for the types III and IV.) 
By \cite[Theorem~7.1]{ET1} we obtain
\begin{eqnarray*}
{\rm rank}\, \Lambda^{(A)}_{\widetilde{f},\widetilde{G}} & = & 2+ \sum_{i=1}^3 \left( \frac{\alpha'_i}{|H_i/G|} -1 \right) \ast |H_i| + \frac{1}{|G_f/G|} a_{W_f},
\end{eqnarray*}
where $\alpha'_i$, $i=1,2,3$, are the Dolgachev numbers of the pair $(f,G_f)$ and $H_i$ is the minimal subgroup of $G_f$ containing $G$ and the isotropy group of the point $p_i$, $i=1,2,3$, of the quotient stack 
\[ 
\mathcal{C}_{(f,G_f)} = [f^{-1}(0) \setminus \{ 0 \}/ \widehat{G}_f ]
\]
(cf.\ \cite{ET1}). This quotient stack is a weighted projective line, see \cite[Proof of Theorem~3]{ET0}. 
Moreover,
\[
\frac{1}{|G_f/G|} a_{W_f}=  \frac{|G/G_0|}{|G_f/G_0|} a_{W_f} = - |G/G_0| \eps_{W_f},
\]
since $|G_f/G_0|=c_f$, and $|G/G_0|$
is the order of the kernel of the degree map $\deg : L_{\widehat{G}} \longrightarrow \ZZ$ (see \cite[Proof of Proposition~17]{ET0}).

By Orlov's semi-orthogonal decomposition theorem \cite[Theorem~2.5]{Or} for $L_{\widehat{G}}$, we have
\[ 
{\rm HMF}^{L_{\widehat{G}}}(f) \cong \langle  D^b {\rm coh}\,\mathcal{C}_{(f,G_f)}, \mathcal{K}(0), \ldots , \mathcal{K}(- |G/G_0| \eps_{W_f}-1) \rangle
\]
where $\mathcal{K}(i)$, $i=0, \ldots , - |G/G_0| \eps_{W_f}-1$, is the smallest full triangulated subcategory containing
the objects isomorphic to $\CC(\vec{l})$ with $\deg(\vec{l}) = i$ in the $L_{\widehat{G}}$-graded singularity category associated to $f$. Now 
\[ 
{\rm rank}\, N(D^b {\rm coh}\,\mathcal{C}_{(f, G_f)}) = 2+ \sum_{i=1}^3 \left( \frac{\alpha'_i}{|H_i/G|} -1 \right) \ast |H_i|,
\]
see \cite[Proposition~4.5]{ET1}. Therefore we obtain
\[
{\rm rank}\, \Lambda^{(A)}_{\widetilde{f}, \widetilde{G}} = {\rm rank}\, N(D^b {\rm coh}\mathcal{C}_{(f,G)}) - |G/G_0| \eps_{W_f} =
{\rm rank}\, \Lambda^{(B)}_{f,G}
\]
what had to be shown.
\end{proof}

Due to H.~Kajiura, Saito,  and the second author \cite{KST1, KST2}, one has the following results:
\begin{theorem}
\begin{itemize}
\item[{\rm (i)}] For $f$ defining an ADE singularity, $\Lambda^{(B)}_{f,G}$ is the root lattice of the corresponding type.
\item[{\rm (ii)}] For $f$ giving a regular system of weights with $\eps_{W_f}=-1$ and $G$ with $j_{\widetilde{G}}=0$, $\Lambda^{(B)}_{f,G}$ is described by a Coxeter--Dynkin diagram of the shape of Fig.\ref{FigT+}, where $\alpha_1, \ldots, \alpha_m$ are the Dolgachev numbers of the pair $(f,G)$.
\end{itemize}
\end{theorem}

Together with Proposition~\ref{prop:8}, this gives rise to the following conjecture:
\begin{conjecture}
Let $f(x_1,x_2, x_3)$ be an invertible polynomial and $G_0 \subset G$. Then we have
\begin{itemize}
\item[{\rm (i)}] The signature of the lattice $\Lambda^{(B)}_{f,G}$ is ${\rm sign}(f,G)$.
\item[{\rm (ii)}] $\Lambda^{(B)}_{f,G} \simeq \Lambda^{(A)}_{\widetilde{f}, \widetilde{G}}$.
\end{itemize}
\end{conjecture}
\begin{remark}
Gusein-Zade and the first author \cite{EG} also introduced an orbifold version of the Milnor lattice modeled on the so-called quantum homology group. 
It is invariant with respect to the orbifold monodromy operator. It is different from our lattices $\Lambda^{(A)}_{\widetilde{f}, \widetilde{G}}$ and $\Lambda^{(B)}_{\widetilde{f}, \widetilde{G}}$. However, in the examples \cite[Examples~1,2,3]{EG}, the invariants $(\mu_+, \mu_0, \mu_-)$ of this lattice agree with the invariants of Proposition~\ref{prop:facts}. (Note that there is a misprint in \cite[Example~2]{EG}: The last entry in the second line of the matrix of the Seifert form $L^{\widetilde{G}}$ must be 4 instead of 2.)
\end{remark}


\end{document}